\documentclass[11pt]{article}%
\usepackage[utf8]{inputenc}
\usepackage{amsfonts}
\usepackage{amsmath}
\usepackage{amssymb}
\usepackage{amsthm}
\usepackage{graphicx}
\usepackage{epstopdf} 
\usepackage{version}
\usepackage{bm}
\usepackage{enumerate}
\usepackage{epsfig}
\usepackage{caption}
\usepackage{setspace}
\usepackage{caption}
\DeclareCaptionLabelFormat{cont}{#1~#2\alph{ContinuedFloat}}
\DeclareMathOperator*{\argmin}{arg\,min}

\DeclareMathOperator\supp{supp}
\newtheorem{theorem}{Theorem}[section]

\newtheorem{lemma}[theorem]{Lemma}

\newtheorem{definition}[theorem]{Definition}
\newtheorem{example}[theorem]{Example}

\newtheorem{remark}[theorem]{Remark}

\numberwithin{equation}{section}
\excludeversion{comment1}

\usepackage{lipsum}

\newcommand\blfootnote[1]{%
  \begingroup
  \renewcommand\thefootnote{}\footnote{#1}%
  \addtocounter{footnote}{-1}%
  \endgroup
}

\begin{document}

\title{Corrected approximation strategy for piecewise smooth bivariate functions}
\date{}
\author{Sergio Amat,
\thanks{
 Departamento de Matem\'atica Aplicada y Estad\'{\i}stica.
   Universidad  Polit\'ecnica de Cartagena (Spain).
e-mail:{\tt sergio.amat@upct.es}}\and
David Levin,
\thanks{
School of Mathematical Sciences. Tel-Aviv University, Tel-Aviv (Israel).
 e-mail:{\tt levindd@gmail.com}
}
\and Juan Ruiz-\'Alvarez \thanks{ Departamento de Matem\'atica Aplicada y Estad\'{\i}stica.
   Universidad  Polit\'ecnica de Cartagena (Spain).
e-mail:{\tt juan.ruiz@upct.es}}
}
\maketitle

\begin{abstract}
Given values of a piecewise smooth function $f$ on a square grid within a domain $\Omega$, we look for a piecewise adaptive approximation to $f$. Standard approximation techniques achieve reduced approximation orders near the boundary of the domain and near curves of jump singularities of the function or its derivatives. The idea used here is that the behavior near the boundaries, or near a singularity curve, is fully characterized and identified by the values of certain differences of the data across the boundary and across the singularity curve.
We refer to these values as the signature of $f$.
In this paper, we aim at using these values in order to define the approximation. That is, we look for an approximation whose signature is matched to the signature of $f$. 
Given function data on a grid, assuming the function is piecewise smooth, first, the singularity structure of the function is identified. For example in the 2-D case, we find an approximation to the curves separating between smooth segments of $f$. Secondly, simultaneously we find the approximations to the different segments of $f$. A system of equations derived from the principle of matching the signature of the approximation and the function with respect to the given grid defines a first stage approximation. An second stage improved approximation is constructed using a global approximation to the error obtained in the first stage approximation.
\end{abstract}
\maketitle
\blfootnote{This work was funded by project 20928/PI/18 (Proyecto financiado por la Comunidad Aut\'onoma de la Regi\'on de Murcia a trav\'es de la convocatoria de Ayudas a proyectos para el desarrollo de investigaci\'on cient\'ifica y t\'ecnica por grupos competitivos, incluida en el Programa Regional de Fomento de la Investigaci\'on Cient\'ifica y T\'ecnica (Plan de Actuaci\'on 2018) de la Fundaci\'on S\'eneca-Agencia de Ciencia y Tecnolog\'ia de la Regi\'on de Murcia) and by the Spanish national research project PID2019-108336GB-I00.}

\section{Introduction}

\medskip
\medskip
Given data values of a piecewise smooth function on a square grid within a domain $\Omega$, one looks for high order approximation to $f$. Standard approximation techniques achieve reduced approximation orders near the boundary of the domain and near curves of jump singularities of the function or its derivatives. In a recent paper \cite{ALR} the authors suggest a novel idea for treating the univariate case using a corrected subdivision approximation. This paper presents a two stage approximation algorithm: In the first stage the data is made smooth by subtracting a proper piecewise polynomial data.
In the second stage a smooth approximation to the smooth data is generated, e.g., by subdivision, and the piecewise polynomial used in the first stage is being added to the smooth approximation to produce the final approximation. Other ways of special treatment near boundaries and close to singularities, also for the univariate case, are reviewed in \cite{ALR}.

The approximation procedure suggested in \cite{ALR} begins with the construction of Taylor series' type approximations to the jumps of the derivarives of $f$ across the singular point $s$. This approach is not easily transferred into the multivariate case where the singularities of the function may occur along curves in the 2D case or across surfaces in the 3D case. We propose here an alternative way for approaching the 1D case, a way which is nalurally transferable into the multivariate case. Within this approach we apply a signature operator which is used both for identifying the singularity location and for deriving a smooth data to be used in the second approximation stage.


\section{The 1-D procedure}

In this section, we present the main idea for univariate function approximation. We have chosen to work with spline approximation, but the same idea can be used for approximation with other basis functions, with similar performance.
We describe the fitting strategy using B-spline basis functions and develop the computation algorithm. 

\subsection{The signature of the data}\hfill

\medskip
Let $f$ be a piecewise smooth function on $[0,1]$, defined by combining two pieces $f_1\in C^m[0,s]$ and $f_2\in C^m(s,1]$.
\begin{equation}\label{f1Dnonsmooth0}
f(x)=
\begin{cases} 
f_1(x)& x< s,\\
f_2(x)& x\ge s. \\ 
\end{cases}
\end{equation}
We are given function values $\{f_i\equiv f(ih)\}_{i=0}^N$, $h=1/(N-1)$. We artificially extend the data outside $[0,1]$ with zero values, $\{f_i=0\}_{i=-k}^{-1}$, $\{f_i=0\}_{i=N+1}^{N+k}$. 

An example is shown in blue in Figure \ref{Ex1Dh001d01}. The process of approximation suggested here requires finding the position $s$ of the discontinuity of the function. As described in \cite{ACDD}, in case of discontinuity at $s$, the interval $(jh,(j+1)h)$ containing $s$ can be detected if $h\le h_c$,
\begin{equation}\label{hc}
h_c :=\frac{|[f]|}{4\sup_{t\in\mathbb{R}\backslash\{s\}} |f^{'}(t)|},
\end{equation}
where $[f]$ is the jump in the function $f$. In the following we assume that this interval is identified.

We suggest that significant information about the function, and especially its behavior near the boundaries and near the singularity point $s$ is encrypted in the sequence of differences of $\bar{f}=\{f_i\}_{i=-k}^{N+k}$. Consider the vector $\bar{d}^k(\bar{f})$ of forward $k$th-order differences of $\bar{f}$, with the elements
\begin{equation}\label{kdifferences}
d^k_i=\Delta^kf_i,\ \ i=-k,...,N.
\end{equation}
These differences are going to be tipically $O(1)$ near $x=0$ and $x=1$ and also near $s$,
and for $k\le m$  they $O(h^k)$ away from the end points and the singular point. Fifth order differences of the data in Figure  \ref{Ex1Dh001d01} are shown in Figure \ref{Diff51Dh01}. Away from the singularity and the end points, the values are of magnitude $\sim 10^{-6}$.
Clearly, the values of the five-order differences show us the critical points for the approximation of $f$. Furthermore, they include important information about the function near the critical points. 

\begin{definition}\label{dkgbar}{\bf The signature of a function - $\sigma_k(g)$}\hfill

Let $g$ be a function on $[0,1]$. We denote by $\bar{g}$ the vector of values of $g$ at the points $\{ih\}_{i=0}^N$, padded by $k$ zero values on each side as above.
We refer to the vector of the forward $k$th-order differences of the data vector $\bar{g}$ as the signature of $g$, and we denote it as $\sigma_k(g)=\bar{d}^k(\bar{g})$, $\sigma_k(g)\in\mathbb{R}^{N+k}$.
\end{definition}


\begin{figure}[!ht]
\begin{center}
    \includegraphics[width=4in]{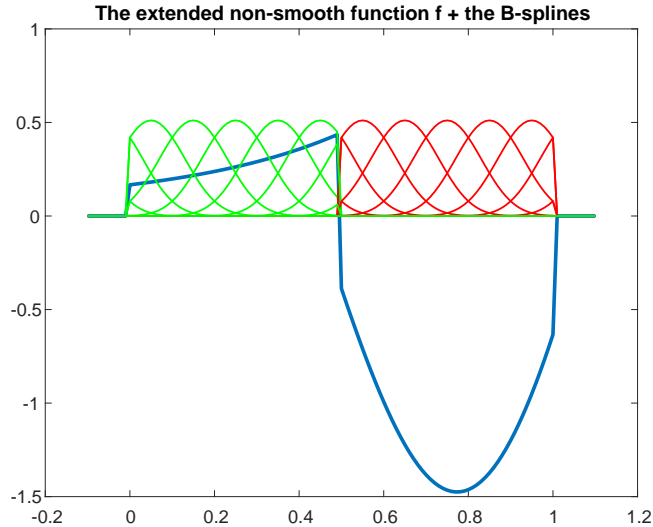}
    \caption{An example of a non-smooth function, with zero padding on both sides, together with the B-spline basis functions used to the right and to the left of the singularity.}
    \label{Ex1Dh001d01}    
\end{center}
\end{figure}

\begin{figure}[!ht]
\begin{center}
    \includegraphics[width=4in]{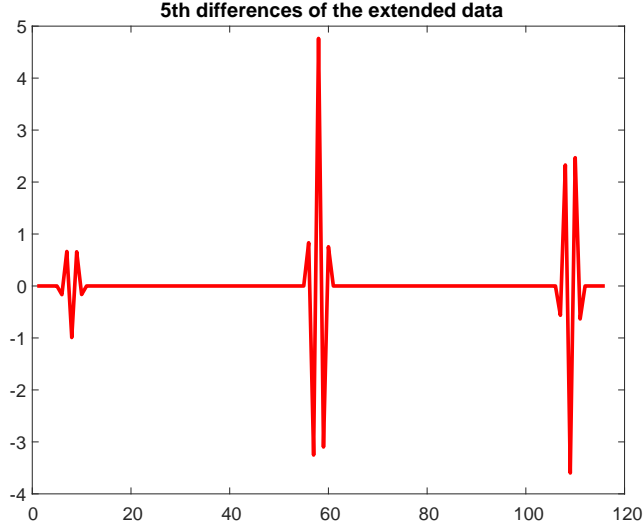}
    \caption{The signature of $f$ using fifth order differences.}
    \label{Diff51Dh01}    
\end{center}    
\end{figure}
\vfill\eject

\subsection{Constructing the first stage approximation}\hfill

\medskip

We choose to build the approximation using $m$th-order spline functions, represented by the B-spline basis.
Let $B^{[m]}_d(x)$ be the B-spline of order $m$ with equidistant knots $\{-md,...,-2d,-d,0\}$. $N_d=1/d+m-1$ is the number of B-splines whose shifts do not vanish in $[0,1]$. The advantage of using spline functions is twofold: 
\begin{itemize}
\item The locality of the B-spline basis functions. 
\item Their approximation power, i.e., if $f\in C^m[a,b]$, there exists a spline $S^{[m]}_d$ such that $\|f-S^{[m]}_d\|_{\infty, [a,b]}\le Cd^{m+1}$.
\end{itemize}
An example of sixth-order
B-splines basis functions used in the numerical example below is shown in Figure \ref{Ex1Dh001d01}. The B-splines used to approximate $f_1$ are shown in green, and the B-splines used to approximate $f_2$ appear in red.

The approach we suggest here involves finding approximations to $f_1$ and $f_2$ simultaneously, using  two separate spline approximations:
\begin{equation}\label{S1}
S_1\equiv S^{[m]}_d|_{[0,s]}(x)=\sum_{i=1}^{N_d}a_{1i}B^{[m]}_d(x-id)|_{[0,s]}\sim f_1,
\end{equation}
and
\begin{equation}\label{S2}
S_2\equiv S^{[m]}_d|_{(s,1]}(x)=\sum_{i=1}^{N_d}a_{2i}B^{[m]}_d(x-id)|_{(s,1]}\sim f_2.
\end{equation}

Following the framework presented in \cite{ALR}, we present below a two stage approximation algorithm: In the first stage the data is made smooth by subtracting a proper piecewise smooth approximation.
In the second stage a smooth approximation to the smooth data is generated and the piecewise approximation used in the first stage is being added to the smooth approximation to produce the final approximation. In this paper the first stage approximation is a piecewise a spline $S$ whose signature, $\sigma_k(S)$, matches the signature of $f$, $\sigma_k(f)$. We define $S$ by a combination of the above $S_1$ and $S_2$.
Since $S$ dependes upon the coefficients $\{a_{1i}\}$ of $S_1$ and the coefficients $\{a_{2i}\}$ of $S_2$, the matching process finds all these unknowns coefficients by solving the minimization problem,

\begin{equation}\label{LS2}
\big[\{a_{1i}\}_{i=1}^{N_d},\{a_{2i}\}_{i=1}^{N_d}\big]=\argmin\|\sigma_k(f)-\sigma_k(S)\|_2^2.
\end{equation}
The minimization is linear w.r.t. the other unknowns. Using the approximation power of $m$th order splines, we can deduce that the minimal value of $\|\sigma_k(f)-\sigma_k(S)\|_2^2$ is $O(Nd^{2m})$.

We denote by $B_{1i}\equiv B^{[m]}_d(\cdot-id)|_{[0,s]}$ the restriction of $B^{[m]}_d(\cdot-id)$ to the interval $[0,s]$, and by  $B_{2i}\equiv B^{[m]}_d(\cdot-id)|_{(s,1]}$ the restriction of $B^{[m]}_d(\cdot-id)$ to the interval $(s,1]$. We concatenate these two sequences of basis functions, $\{B_{1i}\}$ and $\{B_{2i}\}$ into one sequence $\{B_i\}_{i=1}^{2N_d}$, and denote their signatures by $\sigma_i=\sigma_k(\bar{B}_i),\ i=1,..,2N_d$.
The induced system of linear equations for the splines' coefficients $a=(\{a_{1i}\}_{i=1}^{N_d},\{a_{2i}\}_{i=1}^{N_d})$ is $Aa=b$ defined as follows:

\begin{equation}\label{Aijns}
A_{i,j}=\langle \sigma_i,\sigma_j \rangle ,\ \ 1\le i,j \le 2N_d,
\end{equation}
and
\begin{equation}\label{bins}
b_i=\langle \sigma_i,\sigma(f)\rangle ,\ \ 1\le i \le 2N_d.
\end{equation}

\begin{remark}
Due to the locality of the B-splines, some of the basis functions $\{B_{1i}\}$ and $\{B_{2i}\}$ may be identicaly $0$. It thus seems better to use only the non-zero basis functions. From our experience, since we use the general inverse approach for solving the system of equations, using all the basis functions gives the same solution.
As demonstrated in the 2-D procedure below, other approximation spaces  such as tensor product polynomials or trigonometric functions, may be used, with similar approximation results. 
\end{remark}

\begin{remark}
The above construction can be carried out to the case of several singular points.
\end{remark}

\subsubsection{Univariate numerical example}\label{NE1}\hfill

\medskip
We consider the approximation of the function
\begin{equation}\label{Example1f}
f(x)=
\begin{cases} 
\frac{1}{1+(x-1)^2}& x\ge 0.5,\\
\frac{1}{1+(x-1)^2}+(x+1.5)cos(4x)& x< 0.5,\\ 
\end{cases}
\end{equation}
given its values on a uniform grid with $h=0.01$. We have used the signature of $f$ defined by fifth-order forward differences of the extended data  and computed an approximation using fifth-degree splines with equidistant knots' distance $d=0.1$.
For this case, the matrix $A$ is of size $32\times 32$, and $rank(A)=22$. We have solved the linear system using Matlab general inverse procedure pinv, together with an iterative refinement algorithm (described in \cite{Wilk}, \cite{Moler}) to obtain a high precision solution. We denote the resulting piecewise spline approximation by $S^*$.The results are shown in Figure \ref{ExError1Dh001d01} depicting the approximation error $f-S^*$, showing maximal error $\sim 1.75\times 10^{-4}$.

\begin{figure}[!ht]
\begin{center}
    \includegraphics[width=4in]{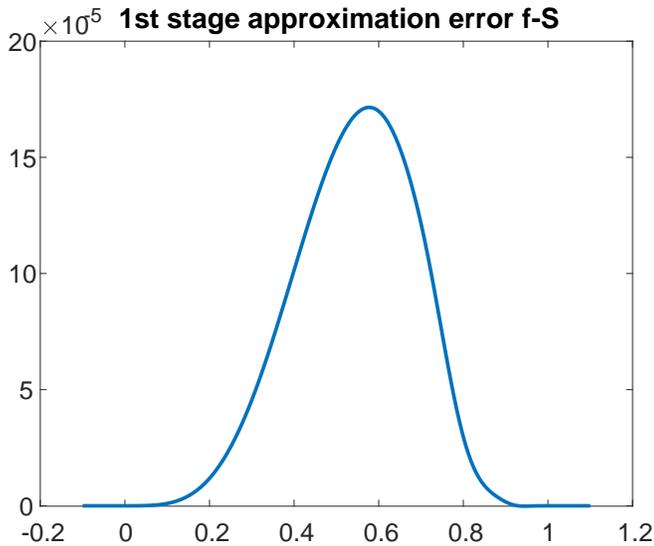}
    \caption{The first stage approximation error $f-S^*$.}
    \label{ExError1Dh001d01}    
\end{center}
\end{figure}

\subsection{Some analysis for the first stage approximation}\label{someanalysis}\hfill

\medskip
Let us estimate the minimal value of $F(S)\equiv \|\sigma_k(f)-\sigma_k(S)\|_2^2$. We do it by finding a function $\tilde{S}$ for which $F(\tilde{S})$ is small. Using the value $F(\tilde{S})$ we can deduce some quantitative estimates for the function $S^*$ which is the minimizer of $F(S)$.

Recalling that $S$ is a piecewise spline function of order $m\ge k$, it follows that both $\sigma_k(f)$ and $\sigma_k(S)$ are $O(h^k)$ away from the boundaries and the singularity, and so is $\sigma_k(f-S)$. Let the knots' distance $d$ be small enough such that the number of B-splines influencing each of the intervals $[0,s]$ and $[s,1]$ is greater that $2k+2$. We can find $\tilde{S}_1$ which coincides with $f_1$ at $k+1$ data points at both ends of the interval $[0,s]$, and $\tilde{S}_2$ which coincides with $f_2$ at $k+1$ data points at both ends of the interval $[s,1]$. Hence, the combination $\tilde{S}$ of $\tilde{S}_1$ and $\tilde{S}_2$ satisfies $\sigma_k(f-\tilde{S})=0$ near $0$, near $s$ and near $1$. Combining this with the observation that $\sigma(f-\tilde{S})\le Ch^k$ away from $0$, $s$ and $1$, it follows that 
\begin{equation}\label{FStilde}
F(\tilde{S})= \|\sigma_k(f-\tilde{S})\|_2^2\le C_0Nh^{2k}.
\end{equation}

Denoting by $S^*$ the minimizer of $\|\sigma_k(f)-\sigma_k(S)\|_2^2$,
\begin{equation}\label{sigmabound}
\|\sigma_k(f-S^*)\|_\infty=\|\sigma_k(f)-\sigma_k(S^*)\|_\infty\le C_1N^{\frac{1}{2}}h^k.
\end{equation}
Thus we have shown that the signature of $S^*$ will be close to the signature of $f$. In order to show that $S^*$ is close to $f$ we need to consider the inverse of the signature operator.
Since $f(ih)=S^*(ih)=0$ for $i=-k,...,-1$, it follows that
$$|f(ih)-S^*(ih)|<C_2N^{\frac{1}{2}}h^k,\ \ 0\le i\le k.$$
Similarly,
$$|f(ih)-S^*(ih)|<C_2N^{\frac{1}{2}}h^k,\ \ N-k\le i\le N.$$
Knowing the error bounds at points near $0$ and near $1$, and in view of
the bound (\ref{sigmabound}), it follows that
\begin{equation}\label{pointbound}
|f(ih)-S^*(ih)|<C_3\|\Sigma^{-1}\|_1N^{\frac{1}{2}}h^k,\ \ 0\le i\le N,
\end{equation}
where $\Sigma$ is the matrix representation of the linear operator defining the signature. For example, for $k=3$ the signature is defined by third-order differences, and the relevant matrix $\Sigma$ is the $N\times N$ 4-diagonal matrix with the elements:
$$\Sigma_{i-2,i}=-1,\ \ \Sigma_{i-1,i}=3,\ \ \Sigma_{i,i}=-3,\ \ \Sigma_{i+1,i}=1.$$
It turns out that  $\|\Sigma^{-1}\|_1=O(N^2)$ which is quite bad to be used in  (\ref{pointbound}), and it calls for an alternative definition of a signature, such that its inverse is bounded independently of $N$. 

Another issue is how close is $S^*$ to $f$ on $[0,1]$. In any case, as we discuss below, we consider $S^*$ as the first approximation to $f$, and the second stage approximation is obtained by correcting the first one.

\subsection{The corrected approximation}

The 1-D procedure developed in \cite{ALR} is a two-stage procedure. In the first stage the non-smooth data is transformed into a smooth data by subtracting a one-sided polynomial with appropriate jumps in its derivatives. The second stage is  a correction step using a standard subdivision process for approximating the residual smooth data.
Here again the approximation $S^*$ obtained by matching the signature of $f$ is just the first stage in computing the final approximation to $f$. The second stage is based upon the observation that the error $e=f-S^*$, evaluated at the data points $\{ih\}_{i=0}^N$, forms a `smooth' data sequence.
Applying an appropriate smooth approximation procedure to the data $\{e(ih)\}_{i=0}^N$, we obtain an approximation $\tilde{e}(x)$ to $e$. The final corrected approximation is defined as
\begin{equation}\label{finalAppr}
\tilde{f}\equiv S^*+\tilde{e}.
\end{equation}
The approximation error $f-\tilde{f}$ is the same as the error $e-\tilde{e}$.
To estimate this error we should examine how smooth is $e$.
If $f$ and its derivatives are discontinuous at $s$, then $e=f-S^*$ would also be discontinuous at $s$. However, as shown below, the smaller the signature of $e=f-S^*$ is made, the smaller is the jump in $e$ and its derivatives across $s$.

To build the approximation to $e$ we may use any univariate approximation method. However, it is simpler to understand the approximation error if we use a local approximation method, such as subdivision or quasi-interpolation by splines.
We assume that the approximation to $e$ is defined as
\begin{equation}\label{appr2e}
\tilde{e}(x)=\sum_{i=-k}^{N+k}e(ih)\phi(x/h-i),
\end{equation}
where $\phi$ is of finite support, $\supp (\phi)=\sigma$, and polynomials of degree $\le \ell$ are reproduced,
\begin{equation}\label{phip}
\sum_{i\in \mathbb{Z}}p(ih)\phi(x/h-i)=p(x),\ \ \ p\in\Pi_\ell.
\end{equation}

If $\ell< m$ ($m$ being the order of the spline $S^*$ and the smoothness degree of $f_1$ and $f_2$), then at points $x$ which are of distance $\ge \sigma$ from $s$ and from $0$ and $1$, we have 
\begin{equation}\label{offs}
|e(x)-\tilde{e}(x)|\le Ch^{\ell+1}.
\end{equation}
Let us check the approximation power near $s$. A similar argument holds near $0$ and $1$. Assuming that $jh\le s<(j+1)h$ we may also assume (by subtracting a polynomial of degree $k-1$) that 
\begin{equation}\label{equal0}
e(ih)=0,\ \ j-k<i\le j.
\end{equation}

Using (\ref{sigmabound}) and (\ref{equal0}) it follows that 
\begin{equation}\label{equaleps}
e(ih)=O(h^{k-\frac{1}{2}}),\ \ j<i\le j+\ell.
\end{equation}
Therefore, both $e$ and $\tilde{e}$ are $O(h^{k-\frac{1}{2}})$ near $s$,
hence near $s$,
\begin{equation}\label{nears}
|e(x)-\tilde{e}(x)|\le Ch^{k-\frac{1}{2}}.
\end{equation}

\subsubsection{Back to the numerial example - The corrected approximation}\hfill

\medskip
Coming back to the numerical example in Section \ref{NE1}, we compute the corrected approximation by applying cubic spline interpolation to the smoothed data and adding $S^*$ to it. While the maximal error in the first stage approximation $S^*$ is $\sim 1.75\times 10^{-4}$, the maximal error in the corrected approximation is $\sim 1.7\times 10^{-9}$.

\begin{figure}[!ht]
\begin{center}
    \includegraphics[width=5in]{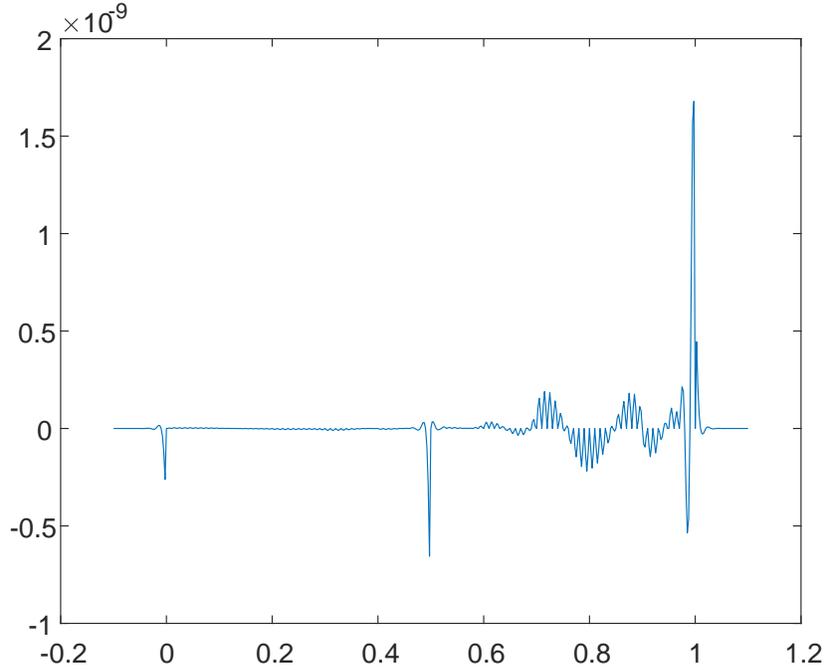}
    \caption{The corrected approximation error $f-\tilde{f}$.}
    \label{ErrorinCorrected1Dh001d01}    
\end{center}
\end{figure}

\section{The 2-D procedure}

Let $f$ be a piecewise smooth function on $[0,1]^2$, defined by the combination of two pieces $f_1\in C^m[\Omega_1]$ and $f_2\in C^m[\Omega_2]$, $\Omega_2=[0,1]^2\setminus \Omega_1$. We assume that we are given function values $\{f_{ij}\equiv f(ih,jh)\}_{i,j=0}^N$, $h=1/(N-1)$. Artificially we extend the data outside $[0,1]^2$ with zero values, $\{f_{ij}=0\}_{i=-k}^{-1}$, $\{f_{ij}=0\}_{i=N+1}^{N+k}$, $\{f_{ij}=0\}_{j=-k}^{-1}$, $\{f_{ij}=0\}_{j=N+1}^{N+k}$.   We do not know the position of the dividing curve separating $\Omega_1$ and $\Omega_2$. We denote this curve by $\Gamma$, and we assume that it is a $C^m$-smooth curve. As in the 1-D case, the existence of a singularity curve in $[0,1]^2$ significantly influences standard approximation procedures,  especially near $\Gamma$, and the approximation power also deteriorates near the boundaries. As in the univariate case, the approximation procedure described below is based upon matching a signature of the function and the approximant. The computation algorithm involves finding approximations to $f_1$ and $f_2$ simultaneously, followed by a correction step. The first stage is the approximation of $\Gamma$.

\subsection{Finding the singularity curve $\Gamma$}\hfill

\medskip
Finding a good approximation of the singularity curve $\Gamma$ is more involved than finding the singularity point $s$ in the univariate case. To simplify the presentation we assume that $\Omega_1$  and $\Omega_2$ are simply connected domains and the set of data points $P=\{(ih,jh)\ | \ 0\le i,j\le N\}$ is consequently divided into two sets:
\begin{equation}\label{P1P2}
\begin{cases}
P_1=\{(ih,jh)\in \Omega_1,\ 0\le i,j\le N\}, \\
P_2=\{(ih,jh)\in \Omega_2,\ 0\le i,j\le N\}. \\
\end{cases}
\end{equation}

As in the univariate case we assume that $h$ is small enough to assure the detection of the singularity interval along each horizontal and vertical line in $[0,1]^2$. This involves estimating a critical $h_c$ as in $(\ref{hc})$ along horizontal and vertical lines (as in \cite{ACDD}).
For each $0\le j\le N$ consider the data values along the horizontal line $y=jh$, $\{f_{ij}\}_{i=0}^N$. As in the univariate case, we find the interval containing the singularity, if exists, and denote the midpoint of this interval as $s_j$. The points
$\{(s_j,jh)\}$ found are at distance $<h/2$ from $\Gamma$. Similarly, we find points along the vertical lines $\{(ih,t_i)\}$ which are near $\Gamma$. An important outcome of this process is that it identifies and separate data points from both sides of the singularity curve. Another way of achieving this is suggested in \cite{AmirLevin}.

Let us continue the description of the procedure alongside the following numerical example:

Consider the piecewise smooth function on $[0,1]^2$ with a jump singularity across the curve $\Gamma$ which is the quarter circle defined by $(x+1)^4+(y+1)^4=10$. The test function is shown in Figure \ref{testf2Dnonsmooth4} and is defined as
\begin{equation}\label{testf2Dnonsmooth}
f(x,y)=
\begin{cases} 
(x+y+2)cos(4x)+sin(4(x+y)),& (x+1)^4+(y+1)^4\ge 10,\\
sin(4(x+y)), & (x+1)^4+(y+1)^4< 10.\\ 
\end{cases}
\end{equation}

\begin{figure}[!ht]
\begin{center}
    \includegraphics[width=5in]{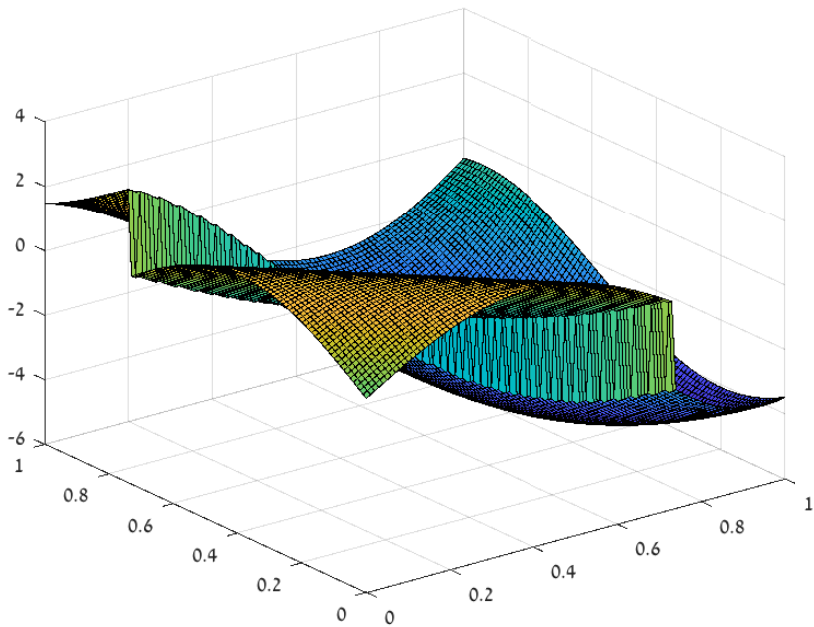}
    \caption{The test function for the 2D non-smooth case.}
    \label{testf2Dnonsmooth4}    
\end{center}
\end{figure}

Let us denote by $Q_0$ all the points $\{(s_j,jh)\}$ and 
$\{(ih,t_i)\}$ found as described above for $h=0.025$. 
We display the points $Q_0$ (in red) in Figure \ref{Max1DiffPoints}, on top of the curve $\Gamma$ (in green).
\begin{figure}[!ht]
\begin{center}
    \includegraphics[width=4in]{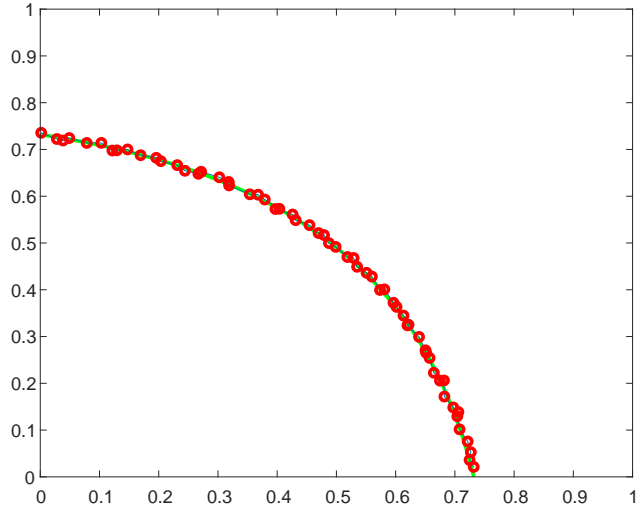}
    \caption{The singularity curve $\Gamma^*$ (green) and the points $Q_0$.}
    \label{Max1DiffPoints}    
\end{center}
\end{figure}
Now we use these points to construct a tensor product cubic spline $D(x,y)$, whose zero level curve defines the approximation to $\Gamma$. To construct $D$ we first overlay on $[0,1]^2$ a net of $m\times m$ points, $Q$, $Q\subset P$. These are the green points displayed in Figure \ref{SignedDistanceAppr4}, for $m=9$.

\begin{figure}[!ht]
\begin{center}
    \includegraphics[width=4in]{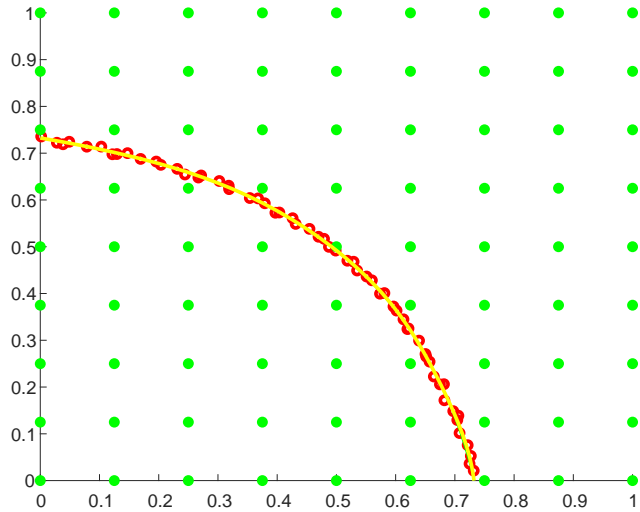}
    \caption{The approximation of the singularity curve $\Gamma$ (yellow).}
    \label{SignedDistanceAppr4}    
\end{center}
\end{figure}

To each point in $Q$ we assign the value of its distance from the set $Q_0$, with a plus sign for the points which are in $P_1$, and a minus sign for the points in $P_2$. To each point in $Q_0$ we assign the value zero. The spline function $D$ is now defined by the least-squares approximation to the values at all the points $Q\cup Q_0$. We have used here tensor product cubic spline on a uniform mesh with knots' distance $d=0.25$. It can be shown that the coresponding normal equations are non-singular if $d>1/m$. We denote the zero level curve of the resulting $D$ as $\Gamma_0$, and this curve is depicted in yellow in Figure \ref{SignedDistanceAppr4}. It seems that 
$\Gamma_0$ is already a good approximation to $\Gamma$ (in green), and yet it may not separate completely the points $P_1$ from the points $P_2$.

\vfill\eject
\subsection{Constructing the approximation}\hfill

As in the univariate case, the approximation strategy is based upon matching some signature of the approximation to the signature of the given function data.

\begin{definition}\label{dkgbar2}{\bf The signature of a function - $\sigma(g)$}\hfill

Let $g$ be a function on $[0,1]^2$. We denote by $\bar{g}$ the matrix of values of $g$ at the points $\{ih,jh)\}_{i=0}^N$, padded by two layers of zero values on each side as above. We define the signature of $g$ as the matrix of discrete biharmonic operator applied to $g$, namely,
\begin{equation}\label{sigma2} 
\sigma(g)=\Delta_h^2(\bar{g}).
\end{equation}
\end{definition}
\medskip

Considering the test function (\ref{testf2Dnonsmooth}), discretized using a mesh size $h=0.01$, padded by ten layers of zero values all arround, its signature is displayed in Figure \ref{sigmatest2D} below.

\begin{figure}[!ht]
\begin{center}
    \includegraphics[width=5in]{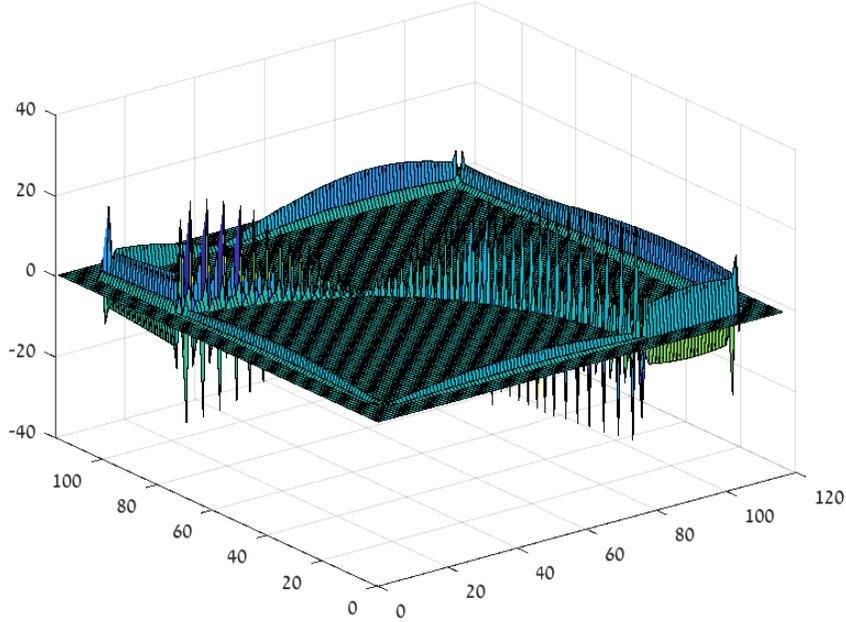}
    \caption{The signature $\sigma(f)$ of the test function (\ref{testf2Dnonsmooth}).}
    \label{sigmatest2D}    
\end{center}    
\end{figure}

The first stage of the approximation process is similar to the construction in the univariate case define by equations (\ref{S1}), (\ref{S2}), (\ref{LS2}). We look here for an approximation $S$ to $f$ which is a combination of two bivariate splines components:
\begin{equation}\label{S12D}
S_1(x,y)=\sum_{i=1}^{N_d}\sum_{j=1}^{N_d}a_{1ij}B_{ij}(x,y), \ \ \ (x,y)\in\tilde{\Omega}_1,
\end{equation}
\begin{equation}\label{S22D}
S_2(x,y)=\sum_{i=1}^{N_d}\sum_{j=1}^{N_d}a_{2ij}B_{ij}(x,y),
\ \ \  (x,y)\in\tilde{\Omega}_2,
\end{equation}
where
\begin{equation}\label{Bij}
B_{ij}(x,y)=B^{[m]}_d(x-id)B^{[m]}_d(y-jd),
\end{equation}
and
\begin{equation}\label{OmegaD1}
\tilde{\Omega}_1=\{(x,y)\ |\ D(x,y)\ge 0,\ (x,y)\in [0,1]^2\ \},
\end{equation}
\begin{equation}\label{OmegaD2}
\tilde{\Omega}_2=\{(x,y)\ |\ D(x,y)< 0,\ (x,y)\in [0,1]^2\ \}.
\end{equation}

\begin{definition}\label{dkgbar3}{\bf The signature of $S$ - $\sigma(S)$}\hfill

We denote by $\bar{S}$ the matrix of values
\begin{equation}\label{sigmaS}
\begin{cases}
\sum_{i=1}^{N_d}\sum_{j=1}^{N_d}a_{1ij}B_{ij}(x,y),\ \ (x,y)\in P_1\ ,\\
\sum_{i=1}^{N_d}\sum_{j=1}^{N_d}a_{2ij}B_{ij}(x,y),\ \ (x,y)\in P_2\ .\\
\end{cases}
\end{equation}
padded by two of layers  zero values on each side as above. The signature of $S$ is $\sigma(S)=\Delta_h^2(\bar{S})$.
\end{definition}

Having defined the signature of $f$ and of $S$, we now look for an approximation $S$
such that the signatures of $f$ and $S$ are matched in the least-squares sense:
\begin{equation}\label{LS2-2D}
\big[\{a_{1ij}\}_{i,j=1}^{N_d},\{a_{2ij}\}_{i,j=1}^{N_d}\big]=\argmin\|\sigma(f)-\sigma(S)\|_2^2.
\end{equation}

\subsubsection{The induced system of equations}\hfill

\begin{definition}\label{dkgbar4}{\bf The signature of the B-splines.}\hfill

For $0\le i,j \le N_d$ let $\bar{B}_{1ij}$ the matrix of $N\times N$ values
\begin{equation}
\begin{cases}
B_{ij}(x,y),\ \ (x,y)\in P_1,\\
0,\ \ (x,y)\in P_2,\\
\end{cases}
\end{equation}
padded by zeros. The signature of $\bar{B}_{1ij}$, $\sigma(\bar{B}_{1ij})$ include the matrix of $\Delta_h^2$ values of $B_{ij}$ on $P_1$. Similarily we define $\sigma(\bar{B}_{2ij})$ using the values of $B_{ij}$ on $P_2$:
\begin{equation}
\begin{cases}
B_{ij}(x,y),\ \ (x,y)\in P_2,\\
0,\ \ (x,y)\in P_1.\\
\end{cases}
\end{equation}
We rearrange the signatures
$\{\sigma(\bar{B}_{1ij})\}_{i,j=1}^{N_d},\{\sigma(\bar{B}_{2ij})\}_{i,j=1}^{N_d}$ as a vector $\sigma(\bar{B})$ of length $2N_d^2$ of signatures.
\end{definition}

\medskip
Rearranging the unknwons $\{a_{1ij}\}_{i,j=1}^{N_d},\{a_{2ij}\}_{i,j=1}^{N_d}$ as a vector $a$ of length $2N_d^2$, and rearranging each signature as a vector of length $N^2$, the linear system of normal equations for the spline coefficients is $Aa=b$, where

\begin{equation}\label{Aijns2}
A_{k,\ell}=\langle \sigma(\bar{B})_k,\sigma(\bar{B})_\ell \rangle ,\ \ 1\le k,\ell \le 2N_d^2,
\end{equation}
and
\begin{equation}\label{bins2}
b_k=\langle \sigma(\bar{B})_k,\sigma(f)\rangle ,\ \ 1\le k \le 2N_d^2.
\end{equation}

The above framework is applied to the numerical example with the test function (\ref{testf2Dnonsmooth}), with $h=0.01$ and sixth-order tensor product splines with knots distance $d=0.1$. Solving the above linear system of equations for the spline coefficients gives us the first stage approximation $S^*$, combined of two segments,
$S^*_1$ on $\tilde{\Omega}_1$ and $S^*_2$ on $\tilde{\Omega}_2$.

In Figure \ref{2D1stApperror}
we show the error $e=f-S^*$ at the grid points. We observe that $e$ is quite smooth, which means that the piecewise spline approximation $S^*$ has captured well the singularity structure of $f$, both across $\Gamma$ and along the boundaries.

\begin{figure}[!ht]
\begin{center}
    \includegraphics[width=5in]{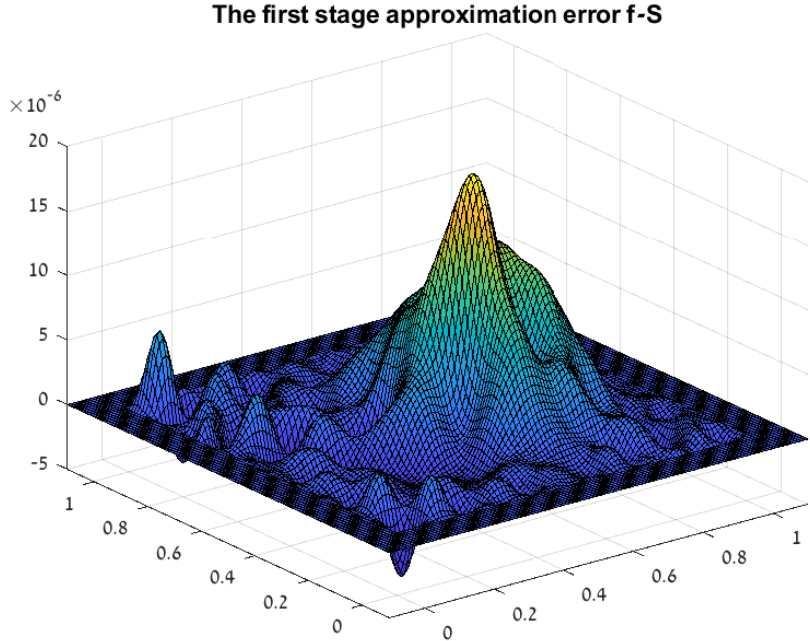}
    \caption{The error in the first stage approximation (\ref{2D1stApperror}).}
    \label{2D1stApperror}    
\end{center}
\end{figure}

\subsection{Second stage - corrected approximation}\label{2Dcorrected}\hfill

\medskip
After computing the first stage approximation $S^*$ obtained by matching the signature of $f$, the second stage is based upon the observation that the error $e=f-S^*$, evaluated at the data points $\{(ih,jh)\}_{i,j=0}^N$, forms a `smooth' data sequence.
Applying an appropriate smooth approximation procedure to the data $\{e(ih,jh)\}_{i,j=0}^N$, we obtain an approximation $\tilde{e}(x)$ to $e$. The final corrected approximation is defined as
\begin{equation}\label{finalAppr}
\tilde{f}\equiv S^*+\tilde{e}.
\end{equation}
Continuing the above numerical example, we have applied fifth order tensor product spline interpolation to $e$. Figure \ref{2D2ndApperror} depicts the values of the error in the final corrected approximation $\tilde{f}$, evaluated on a fine grid. We observe that the absolute value of the maximal error is $2.8\times 10^{-7}$, and it is attained  at the boundary. There are also errors of magnitude $\sim 10^{-8}$ near the singularity curve, but everywhere else the errors are of magnitude $\sim 10^{-10}$.
We recall that $f$ is piecewise-defined on $\Omega_1$ and $\Omega_2$, while $S^*$ is piecewise-defined on $\tilde{\Omega}_1$ and $\tilde{\Omega}_2$. In comparing $f$ and $\tilde{f}$, we let $S^*$ be piecewise-defined on $\Omega_1$ and $\Omega_2$, with the same $S^*_1$ and $S^*_2$. Similar results are obtained if we let $f$ be  piecewise-defined on $\tilde{\Omega}_1$ and $\tilde{\Omega}_2$, with the same $f_1$ and $f_2$.

\begin{figure}[!ht]
\begin{center}
    \includegraphics[width=5in]{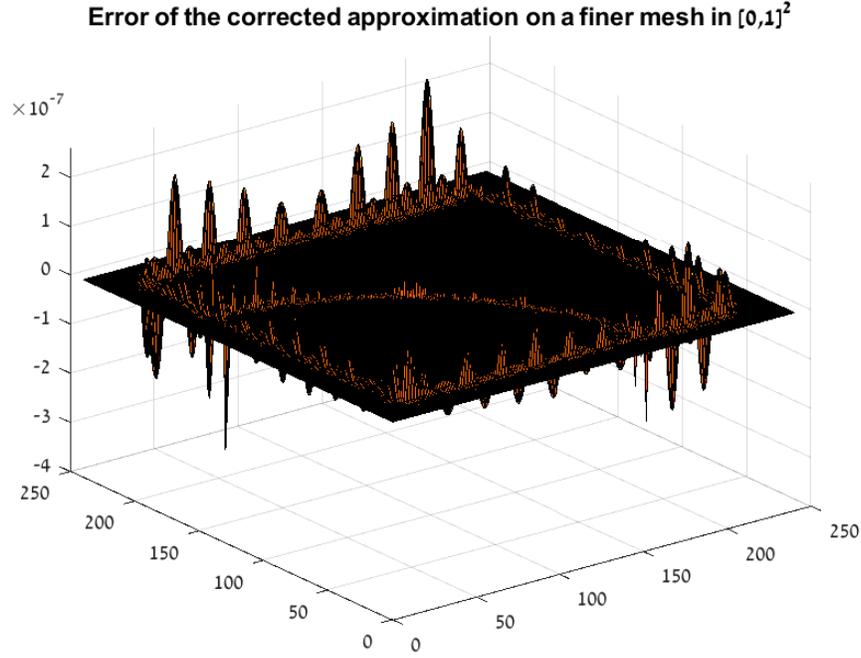}
    \caption{The error in the second stage corrected approximation $f-\tilde{f}$.}
    \label{2D2ndApperror}    
\end{center}
\end{figure}

\subsection{Second numerical example - Two singularity curves}\label{2curves}\hfill

\medskip
Consider the case of two disjoint singularity curves, subdividing the domain into three sub-domains, $\Omega_1,\Omega_2,\Omega_3$, separated by two smooth curves $\Gamma_1$ and $\Gamma_2$, and a piecewise-defined function as in Figure \ref{testf2curves}.

\begin{figure}[!ht]
\begin{center}
    \includegraphics[width=4in]{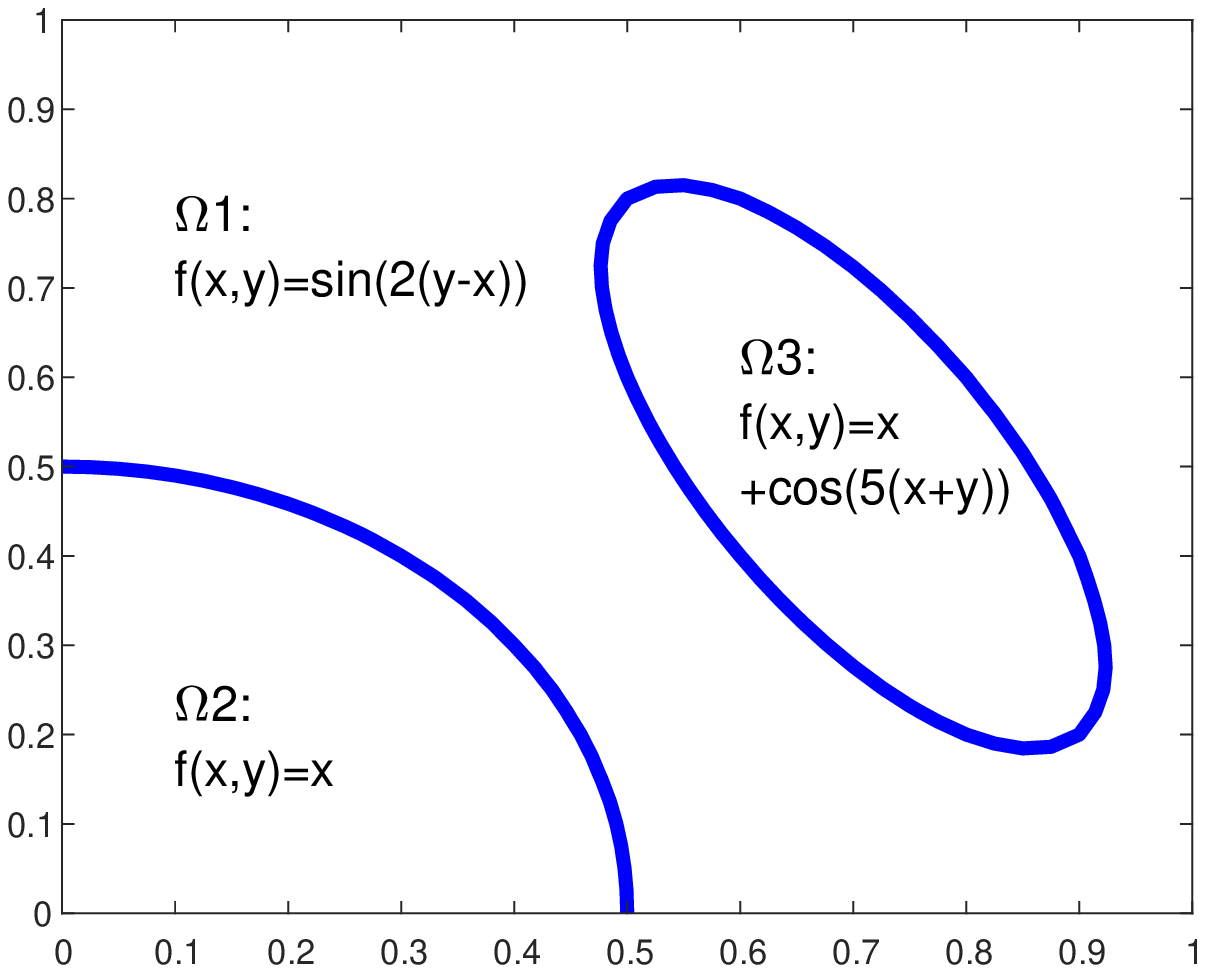}
    \caption{The test function for Example \ref{2curves}.}
    \label{testf2curves}    
\end{center}
\end{figure}

Function values are given on a uniform grid with $h=0.01$. In Figure \ref{testf2D2curvesh001d01} we display the extended data, padded with zeros arround the square. 

\begin{figure}[!ht]
\begin{center}
    \includegraphics[width=4in]{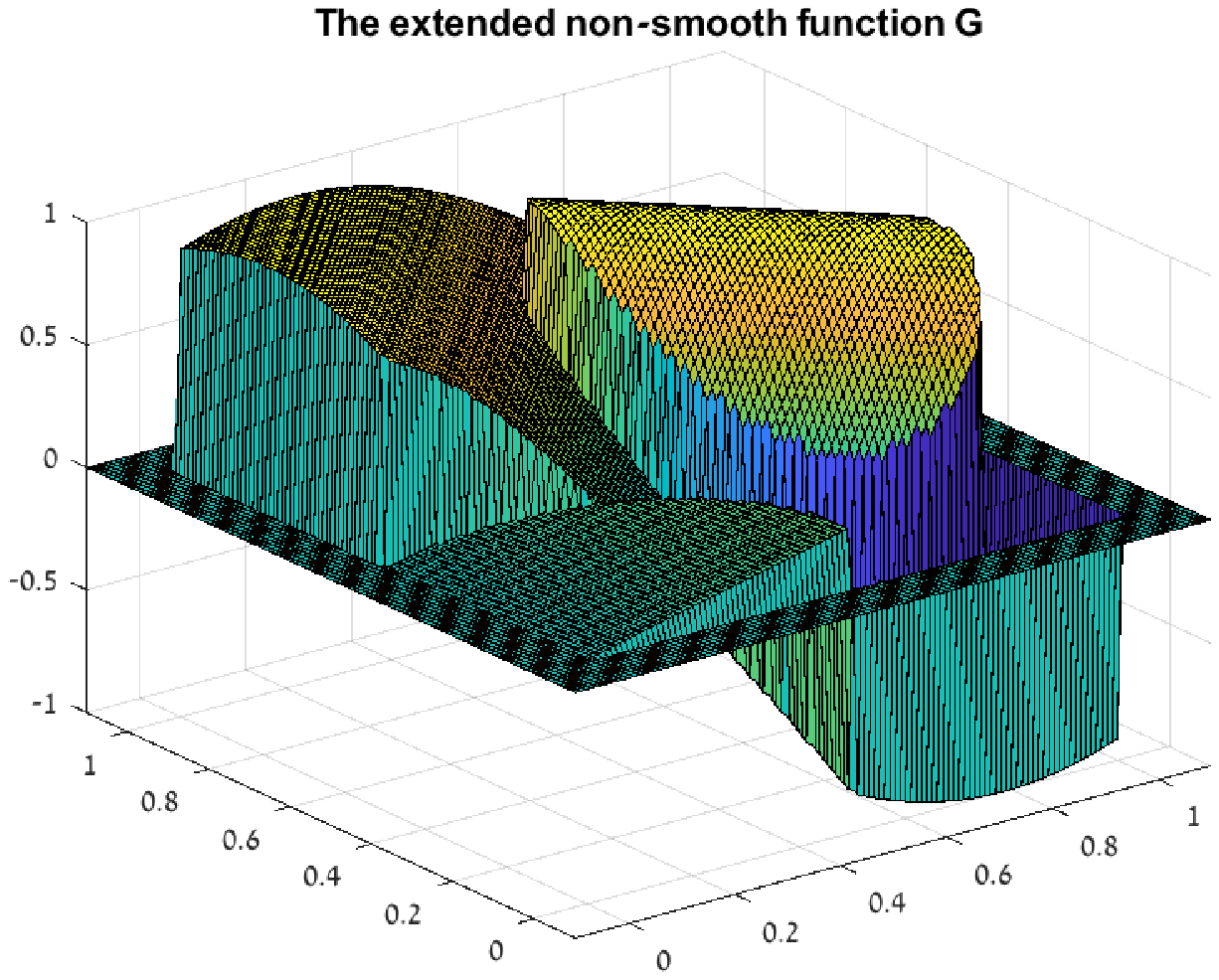}
    \caption{The test function for Example \ref{2curves}.}
    \label{testf2D2curvesh001d01}    
\end{center}
\end{figure}

Function values are given on a uniform grid with $h=0.01$. In Figure \ref{testf2D2curvesh001d01} we display the extended data, padded with zeros arround the square. 

\subsubsection{Approximating the singularity curves}

As in the first example, we assume that $h$ is small enough to assure the detection of the singularity interval along each horizontal and vertical line in $[0,1]^2$ (\cite{ACDD}).
For each horizontal line $y=jh$ and each vertical line $x=ih$ we detect the intervals containing a singular point, and collect all the midpoints of these intervals, denoting this set of points as $Q_0$. Within the detection algorithm we also include a proper ordering algorithm by which we obtain a subdivision of the set of data points $P=\{(ih,jh)\}_{i,j=0}^N$ into three disjoint sets, $P_k\in \Omega_k$, $k=1,2,3$. $P_1$ denotes the set which has neighbors in the two other sets, to be denoted as $P_2$ and $P_3$.

As in the case of one singularity curve, we construct a tensor product cubic spline $D(x,y)$, whose zero level curves approximate $\Gamma_1$ and $\Gamma_2$. To construct $D$ we overlay on $[0,1]^2$ a net of $11\times 11$ points, $Q$, $Q\subset P$.
\begin{figure}[!ht]
\begin{center}
    \includegraphics[width=4in]{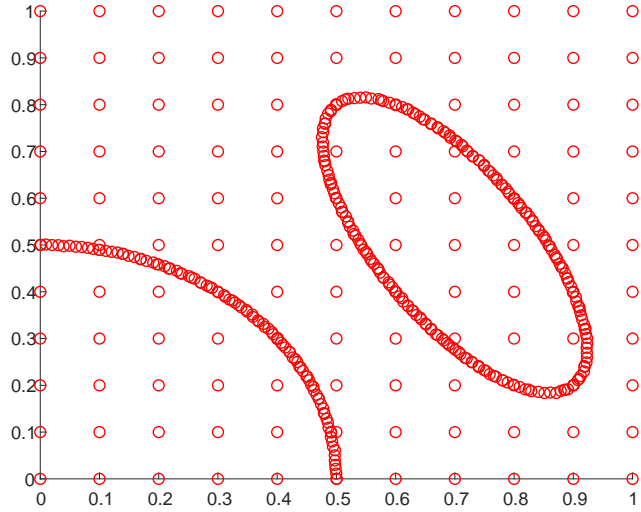}
    \caption{The points $Q\cup Q_0$ used for approximating the singularity curves.}
    \label{QQ02curves}    
\end{center}
\end{figure}
To each point in $Q$ we assign the value of its distance from the set $Q_0$, with a plus sign for the points which are in $P_1$, and for the points in $P_2$ or $P_3$. To each point in $Q_0$ we assign the value zero. The spline function $D$ is now defined by the least-squares approximation to the values at all the points $Q\cup Q_0$ displayed in Figure \ref{QQ02curves}. We have used here tensor product cubic spline on a uniform mesh with knots' distance $d=0.2$. We denote the zero level curves of the resulting $D$ as $\tilde{\Gamma}_1$, $\tilde{\Gamma}_2$, and these curves are depicted in yellow in Figure \ref{SignedDistanceAppr2cirves}. These curves provide a good approximation to $\Gamma_1$ and $\Gamma_2$. To improve the approximation quality we assign higher weight ($\times 100$) to the points $Q_0$ in the least-squares cost function defining $D$.

\begin{figure}[!ht]
\begin{center}
    \includegraphics[width=4in]{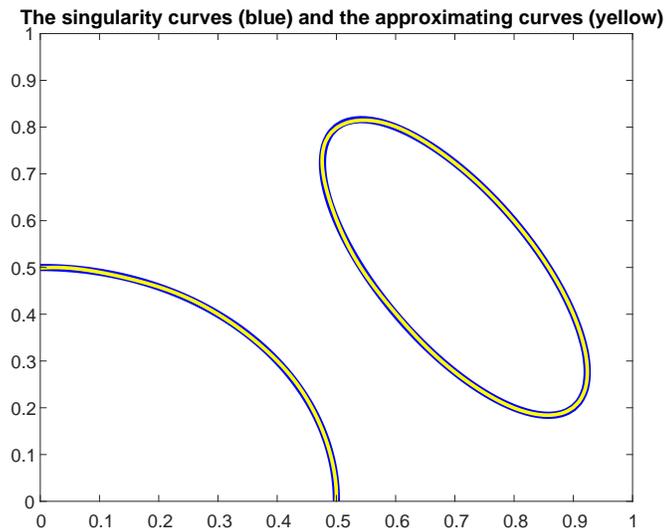}
    \caption{The singularity curves for Example \ref{2curves}, in blue, and their approximations, in yellow.}
    \label{SignedDistanceAppr2cirves}    
\end{center}
\end{figure}

Considering the test function defined in Figure \ref{testf2curves}, discretized using mesh size $h=0.01$ padded by ten layers of zero values all arround, its $\Delta_h^2$ signature is displayed in Figure \ref{sigmatestf2curves} below.

\begin{figure}[!ht]
\begin{center}
    \includegraphics[width=5in]{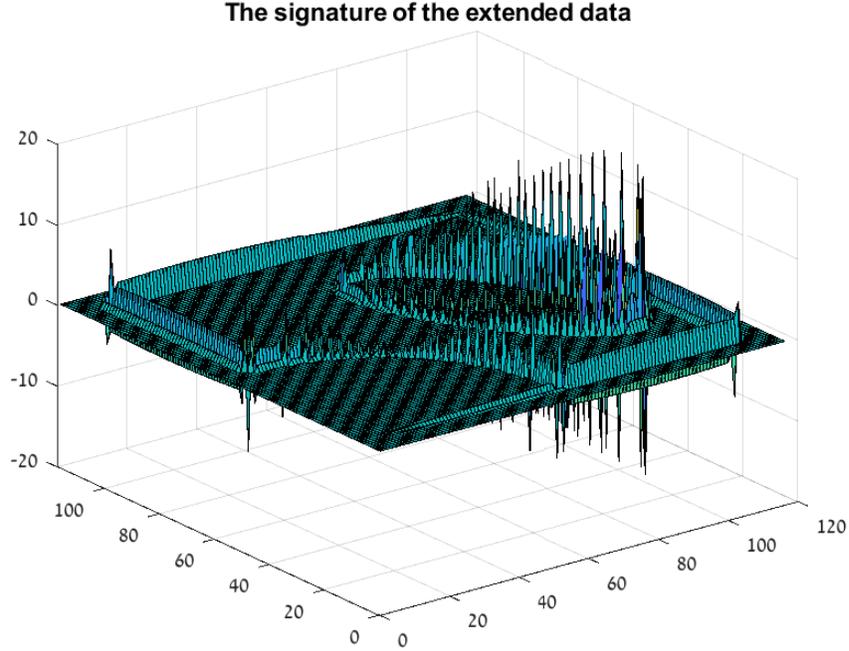}
    \caption{The signature $\sigma(f)$ of the test function with two singularity curves.}
    \label{sigmatestf2curves}    
\end{center}    
\end{figure}

For the first stage approximation process we look here for an approximation $S$ to $f$ which is a combination of three bivariate splines components:
\begin{equation}\label{S12D2}
S_1(x,y)=\sum_{i=1}^{N_d}\sum_{j=1}^{N_d}a_{1ij}B_{ij}(x,y), \ \ \ (x,y)\in\tilde{\Omega}_1,
\end{equation}
\begin{equation}\label{S22D2}
S_2(x,y)=\sum_{i=1}^{N_d}\sum_{j=1}^{N_d}a_{2ij}B_{ij}(x,y),
\ \ \  (x,y)\in\tilde{\Omega}_2,
\end{equation}
\begin{equation}\label{S32D3}
S_3(x,y)=\sum_{i=1}^{N_d}\sum_{j=1}^{N_d}a_{3ij}B_{ij}(x,y),
\ \ \  (x,y)\in\tilde{\Omega}_3,
\end{equation}

where $\tilde{\Omega}_1$, $\tilde{\Omega}_2$, $\tilde{\Omega}_3$ are the three disjoint domains defined by the bicubic spline $D$.

\begin{definition}\label{dkgbar4}{\bf The signature of $S$ - $\sigma(S)$}\hfill
As in Definition \ref{dkgbar3} above, we define $\bar{S}$ as the matrix of values of $S$ at the grid points,
padded by two layers of zero values all arround, and compute its signature $\sigma(S)=\Delta_h^2(\bar{S})$.
\end{definition}

Having defined the signature of $f$ and of $S$, we now look for an approximation $S$
such that the signatures of $f$ and $S$ are matched in the least-squares sense:
\begin{equation}\label{LS2-2D}
\big[\{a_{1ij}\}_{i,j=1}^{N_d},\{a_{2ij}\}_{i,j=1}^{N_d},\{a_{3ij}\}_{i,j=1}^{N_d}\big]=\argmin\|\sigma(f)-\sigma(S)\|_2^2.
\end{equation}

\subsubsection{The induced system of equations}\hfill

For $r=1,2,3$, the signature $\sigma(\bar{B}_{rij})$ include the matrix of $\Delta_h^2$ values of $B_{ij}$ restricted to $P_r$. 
We rearrange the signatures
$\{\sigma(\bar{B}_{rij})\}_{i,j=1}^{N_d}$, $r=1,2,3$, as a vector $\sigma(\bar{B})$ of length $3N_d^2$ of signatures.

\medskip
Rearranging the unknwons $\{a_{rij}\}_{i,j=1}^{N_d}$, $r=1,2,3$, as a vector $a$ of length $3N_d^2$, the linear system of normal equations for the spline coefficients is $Aa=b$, where

\begin{equation}\label{Aijns2curves}
A_{k,\ell}=\langle \sigma(\bar{B})_k,\sigma(\bar{B})_\ell \rangle ,\ \ 1\le k,\ell \le 3N_d^2,
\end{equation}
and
\begin{equation}\label{bins2curves}
b_k=\langle \sigma(\bar{B})_k,\sigma(f)\rangle ,\ \ 1\le k \le 3N_d^2.
\end{equation}

The above framework is applied to the numerical example with the test function defined in Figure \ref{testf2curves}, with $h=0.01$ and sixth-order tensor product splines with knots distance $d=0.1$. Solving the above linear system of equations for the spline coefficients gives us the first stage approximation $S^*$, combined of two segments,
$S^*_1$ on $\tilde{\Omega}_1$, $S^*_2$ on $\tilde{\Omega}_2$ and $S^*_3$ on $\tilde{\Omega}_3$.

In Figure \ref{2D2curves1stApperror}
we show the error $e=f-S^*$ at the initial grid points. We observe that $e$ is quite smooth, which means that the piecewise spline approximation $S^*$ has captured well the singularity structure of $f$, both across the singularity curves and along the boundaries.

\begin{figure}[!ht]
\begin{center}
    \includegraphics[width=5in]{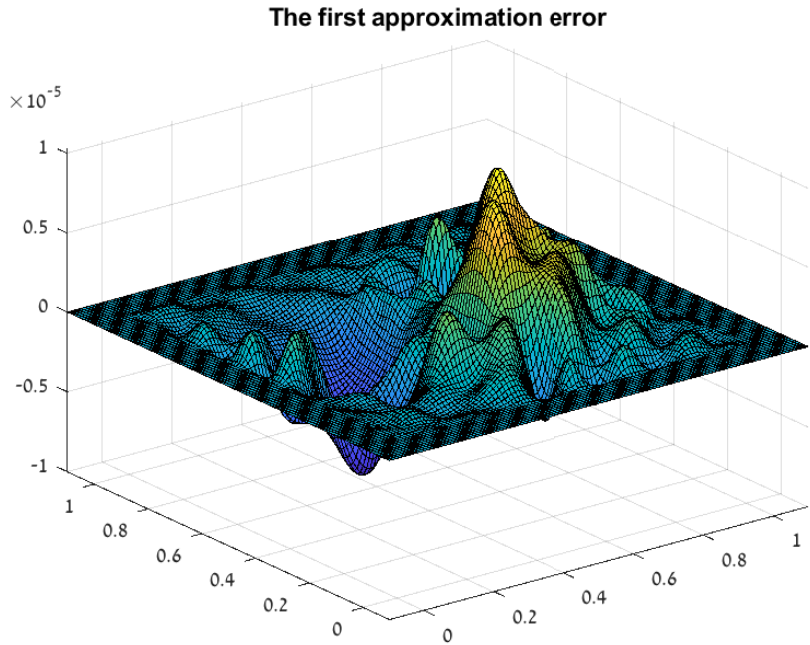}
    \caption{The error in the first stage approximation to $f$.}
    \label{2D2curves1stApperror}
\end{center}    
\end{figure}

\subsection{Second stage - corrected approximation}\label{2Dcorrected}\hfill

\medskip
Here also the error $e=f-S^*$, evaluated at the data points $\{(ih,jh)\}_{i,j=0}^N$, forms a `smooth' data sequence, and we apply smooth approximation procedure to the data $\{e(ih,jh)\}_{i,j=0}^N$ to obtain an approximation $\tilde{e}(x)$. The final corrected approximation is defined as
\begin{equation}\label{finalAppr}
\tilde{f}\equiv S^*+\tilde{e}.
\end{equation}
Computing  $\tilde{e}$ using a fifth order tensor product spline interpolation to $e$, we derive the final corrected approximation to $f$, shown in Figure \ref{Final2curves}.

\begin{figure}[!ht]
\begin{center}
    \includegraphics[width=5in]{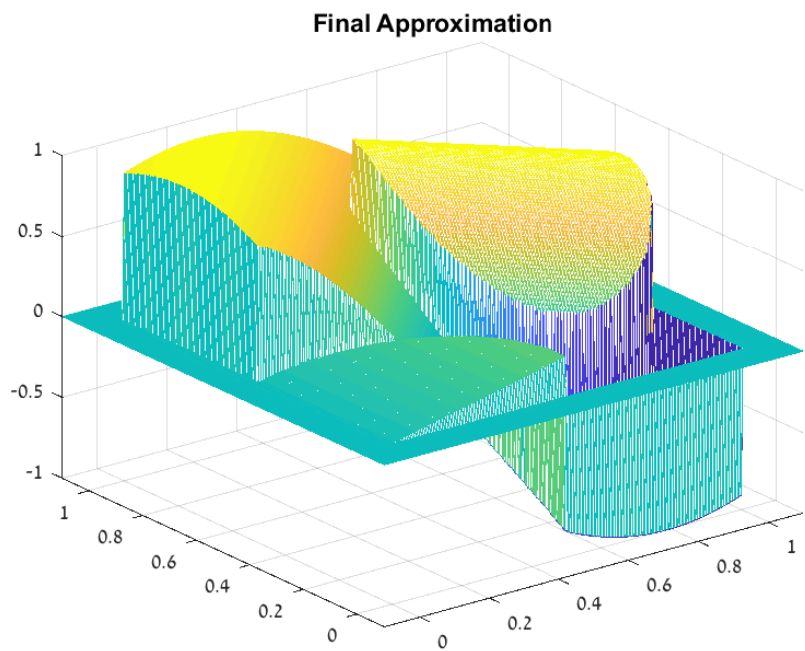}
    \caption{The final corrected approximation to $f$.}
    \label{Final2curves}
\end{center}    
\end{figure}

Figure \ref{Final2curvesApperror} depicts the values of the error in the final corrected approximation $\tilde{f}$, evaluated on a fine grid. We observe that the absolute value of the maximal error is $2\times 10^{-7}$, and it is attained  at a singularity curve. There are also errors of magnitude $\sim 10^{-8}$ near the boundaries, but everywhere else the errors are of magnitude $\sim 10^{-12}$.

\begin{figure}[!ht]
\begin{center}
    \includegraphics[width=5in]{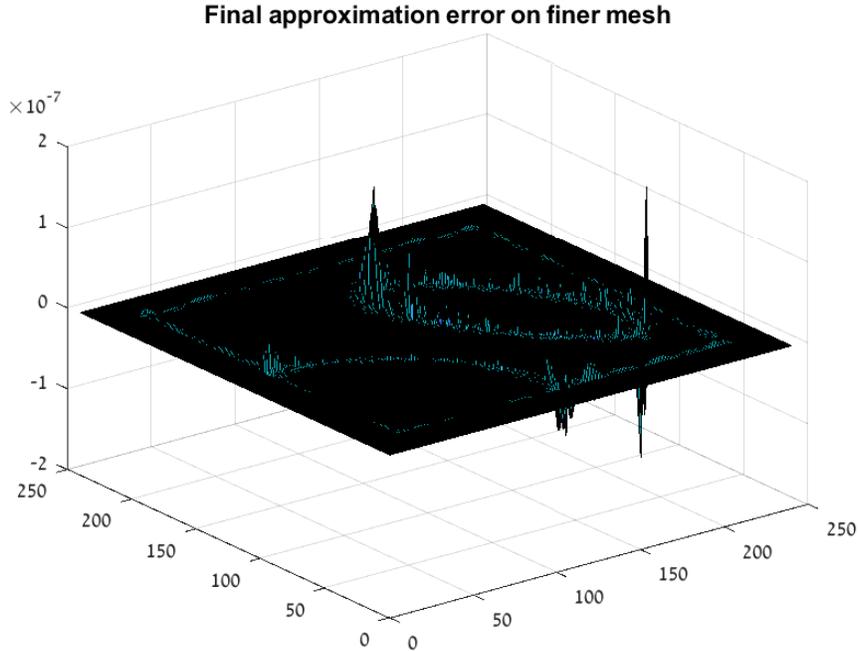}
    \caption{The final corrected approximation errors.}
    \label{Final2curvesApperror}
\end{center}    
\end{figure}

\vfill\eject
\section{A 3-D example}

Let $f$ be a piecewise smooth function on $[0,1]^3$, defined by the combination of two pieces $f_1\in C^m[\Omega_1]$ and $f_2\in C^m[\Omega_2]$, $\Omega_2=[0,1]^3\setminus \Omega_1$. We assume that we are given function values $\{f_{ijk}\equiv f(ih,jh,kh)\}_{i,j,k=0}^N$, $h=1/(N-1)$. Artificially we extend the data outside $[0,1]^3$ with zero values. We do not know the position of the dividing surface separating $\Omega_1$ and $\Omega_2$. We denote this surface by $\Theta$, and we assume that it is a $C^m$-smooth surface. As in the 1-D case, the existence of a singularity in $[0,1]^3$ significantly influences standard approximation procedures,  especially near $\Theta$, and the approximation power also deteriorates near the boundaries. As in the univariate case, the approximation procedure described below is based upon matching a signature of the function and the approximant. The computation algorithm involves finding approximations to $f_1$ and $f_2$ simultaneously, followed by a correction step.

The first step of approximating $\Theta$ is a straightforward extension of the 2-D procedure for approximating singularity curves.

Consider the piecewise smooth function $f(x,y)$ on $\Omega=[0,1]^3$ with a jump singularity across the surface $\Theta$ which is the surface of the closed 4-ball $B$ of radius $0.33$ centered at $(0.5,0.5,0.5)$,
$$B=\{(x,y,z)\ |\ (x-0.5)^4+(y-0.5)^4+(z-0.5)^4\le 0.33^4\},$$

\begin{equation}\label{testf3Dnonsmooth}
f(x,y,z)=
\begin{cases} 
(exp((x+y)/2)-1)sin(2x),& (x,y,z)\notin B,\\
sin(4(x^2+y^2+z^2))sin(2(x-y)), & (x,y,z)\in B.\\ 
\end{cases}
\end{equation}
Assume we are given function values $\{f_{ijk}=f(ih,jh,kh),\ (ih,jh,kh)\in \Omega\}$, $h=1/(N-1)$,
and artificially extend the data with zero values at $m$ layers around $\Omega$ .
A cross-section of the extended data of the test function (\ref{testf3Dnonsmooth}), for $z=0.5$, $N=41$, is shown in Figure \ref{testf3Dh025B4}.

 \begin{figure}[!ht]
 \begin{center}
    \includegraphics[width=5in]{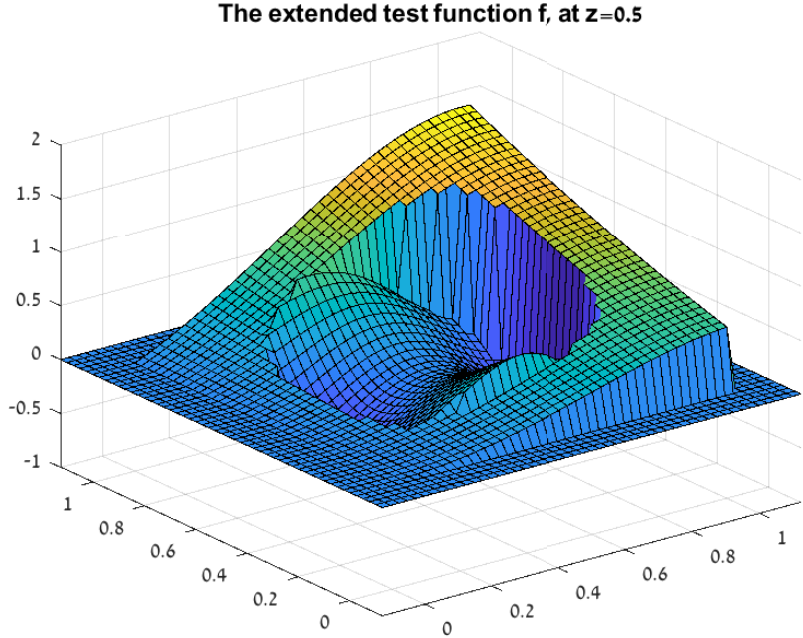}
    \caption{A cross-section of the 3D test function.}
    \label{testf3Dh025B4}
\end{center}
\end{figure}

Let us denote by $Q_0$ all the points $\{(s_{j,k},jh,kh)\}$,
$\{(ih,t_{i,k},kh)\}$, and $\{(ih,jh,r_{i,j})\}$ found as described above for the 2D case, i.e., the midpoints of intervals containing a singularity along lines parallel to the axes. We use these points to construct a tensor product cubic spline $D(x,y,z)$, whose zero level surface defines the approximation to $\Theta$. To construct $D$ we first overlay on $[0,1]^3$ a net of $m\times m$ points, $Q$. These are the red points displayed in Figure \ref{3DsurfacedistB4}, for $m=11$, together with 
the points $Q_0$, for $h=0.025$, in blue.
\begin{figure}[!ht]
\begin{center}
    \includegraphics[width=4in]{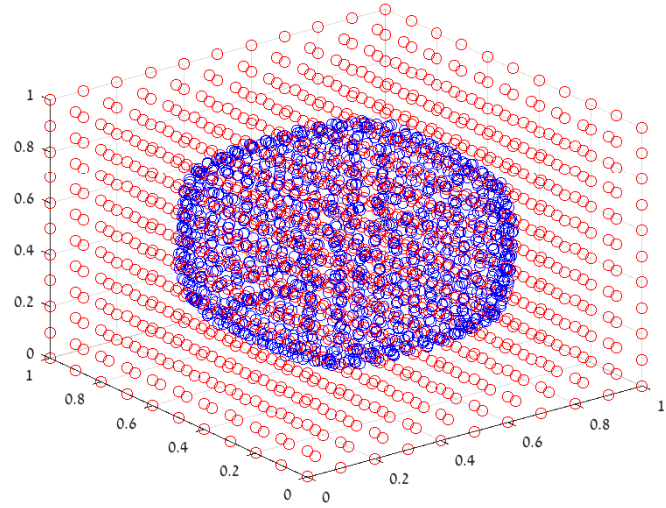}
    \caption{The points $Q$ (red) and the points $Q_0$ (blue).}
    \label{3DsurfacedistB4}
\end{center}    
\end{figure}

\begin{figure}[!ht]
\begin{center}
    \includegraphics[width=4in]{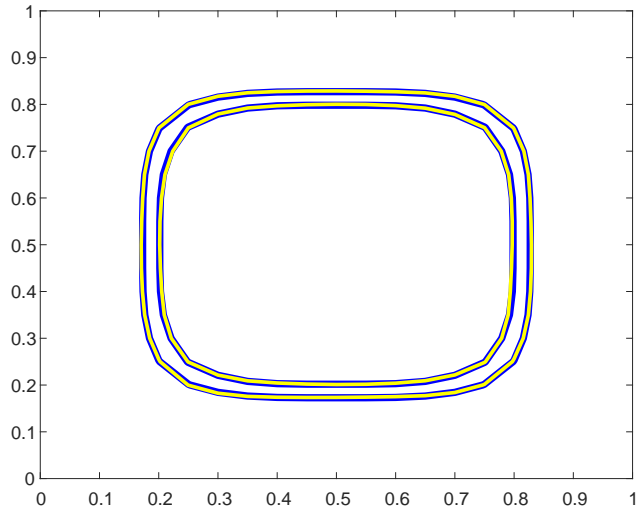}
    \caption{Cross-sections of the approximation of the singularity surface $\Theta$ (yellow), at $z=0.25$ and $z=0.5$.}
    \label{3DsurfaceappB4}
\end{center}
\end{figure}

To each point in $Q$ we assign the value of its distance from the set $Q_0$, with a plus sign for the points which are in $B$, and a minus sign for the points outside $B$. To each point in $Q_0$ we assign the value zero. The spline function $D$ is now defined by the least-squares approximation to the values at all the points $Q\cup Q_0$. We have used here tensor product cubic spline on a uniform mesh with knots' distance $d=0.2$. We denote the zero level surface of the resulting $D$ as $\Theta_0$, and cross-section curves of this surface are depicted in yellow in Figure \ref{3DsurfaceappB4}, together with the exact relevant curves of $\Theta$ in blue. 

Next, we show here some numerical results of the first approximation step and the correction step for a specific numerical example. Two new ingredients we choose to present here: Using a different signature operator, and using Tchebyshev polynomials instead of B-splines for building the first stage approximation.

Here we have used the signature of $f$ defined by fourth-order forward differences of the extended data in three directions. The first stage approximation is computed by matching the signature of $f$ and the signature of the approximant $S$ defined by two sets of tensor product Tchebyshev polynomials up to degree $n$. One for the approximation over $\Omega_1=B$, and the other for approximation over $\Omega_2=[0,1]^3\setminus B$. 
In 3-D numerical tests we are limited by the memory constraints of Matlab.
We took $h=0.025$ and tensor product polynomials of degree $8$. The error in the resulting approximation is displayed in Figure \ref{3D1stapperrorB4}.
 
  \begin{figure}[!ht]
  \begin{center}
    \includegraphics[width=5in]{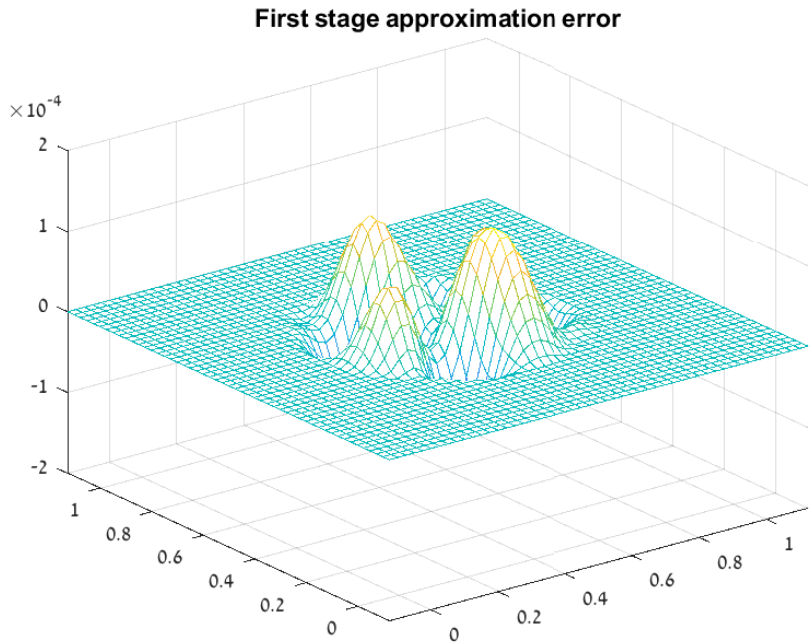}
    \caption{First stage approximation error, at $z=0.5$.}
    \label{3D1stapperrorB4}
\end{center}    
\end{figure}

\medskip

\begin{remark}{\bf 3D analysis}\hfill

\medskip
Let the signature $\sigma_m$ of the function data be defined by its $m$th order differences along the 3 axes.
Adapting the discussion in Section \ref{someanalysis} to the 3D case, it follows that for the approximant $S^*$ minimizing
$\|\sigma_m(f)-\sigma_m(S)\|_2^2$, it follows that
\begin{equation}\label{sigmabound3}
\|\sigma_m(f-S^*)\|_\infty=\|\sigma_m(f)-\sigma_m(S^*)\|_\infty\le C_1N^{\frac{3}{2}}h^m.
\end{equation}
\end{remark}

Using the above result, it follows that
the error $e=f-S^*$, evaluated at the data points $\{(ih,jh,kh)\}_{i,j,k=0}^N$, forms a `smooth' data sequence, and we can apply a smooth approximation procedure to the data $\{e(ih,jh,kh)\}_{i,j,k=0}^N$ to obtain an approximation $\tilde{e}(x)$. The final corrected approximation is defined as
\begin{equation}\label{finalAppr}
\tilde{f}\equiv S^*+\tilde{e}.
\end{equation}
Computing  $\tilde{e}$ using cubic tensor product spline interpolation to $e$, we derive the final corrected approximation to $f$, shown in Figure \ref{3D2ndfineappB4}.

\begin{figure}[!ht]
\begin{center}
    \includegraphics[width=5in]{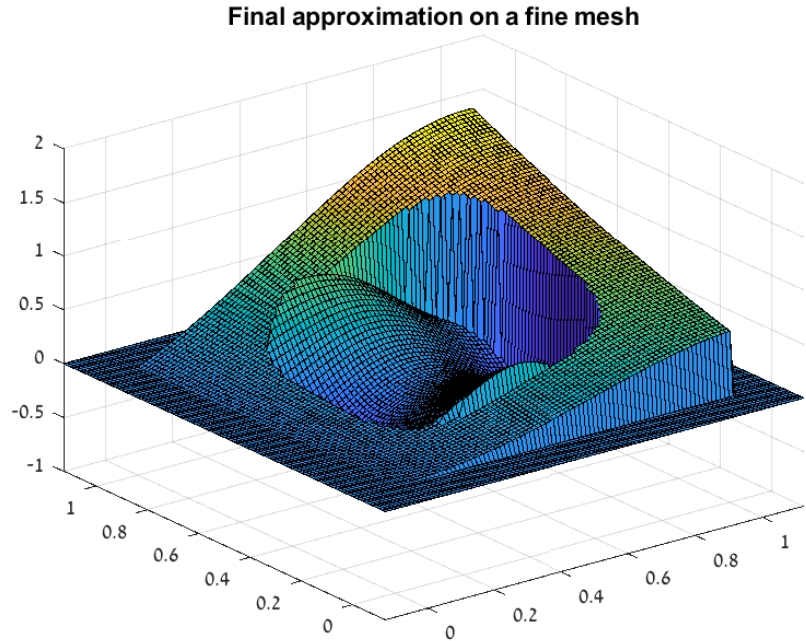}
    \caption{The final corrected approximation on a fine mesh, at $z=0.5$.}
    \label{3D2ndfineappB4}
\end{center}
\end{figure}

Figure \ref{3D2ndapperrorB4} depicts the values of the error in the final corrected approximation $\tilde{f}$, evaluated on a fine grid. We observe that the absolute value of the maximal error is $6.5\times 10^{-6}$, and it is attained  at a singularity surface. There are also errors of magnitude $\sim 5\times 10^{-8}$ near the boundaries, and everywhere else the errors are of magnitude $\sim 10^{-9}$.

\begin{figure}[!ht]
\begin{center}
    \includegraphics[width=5in]{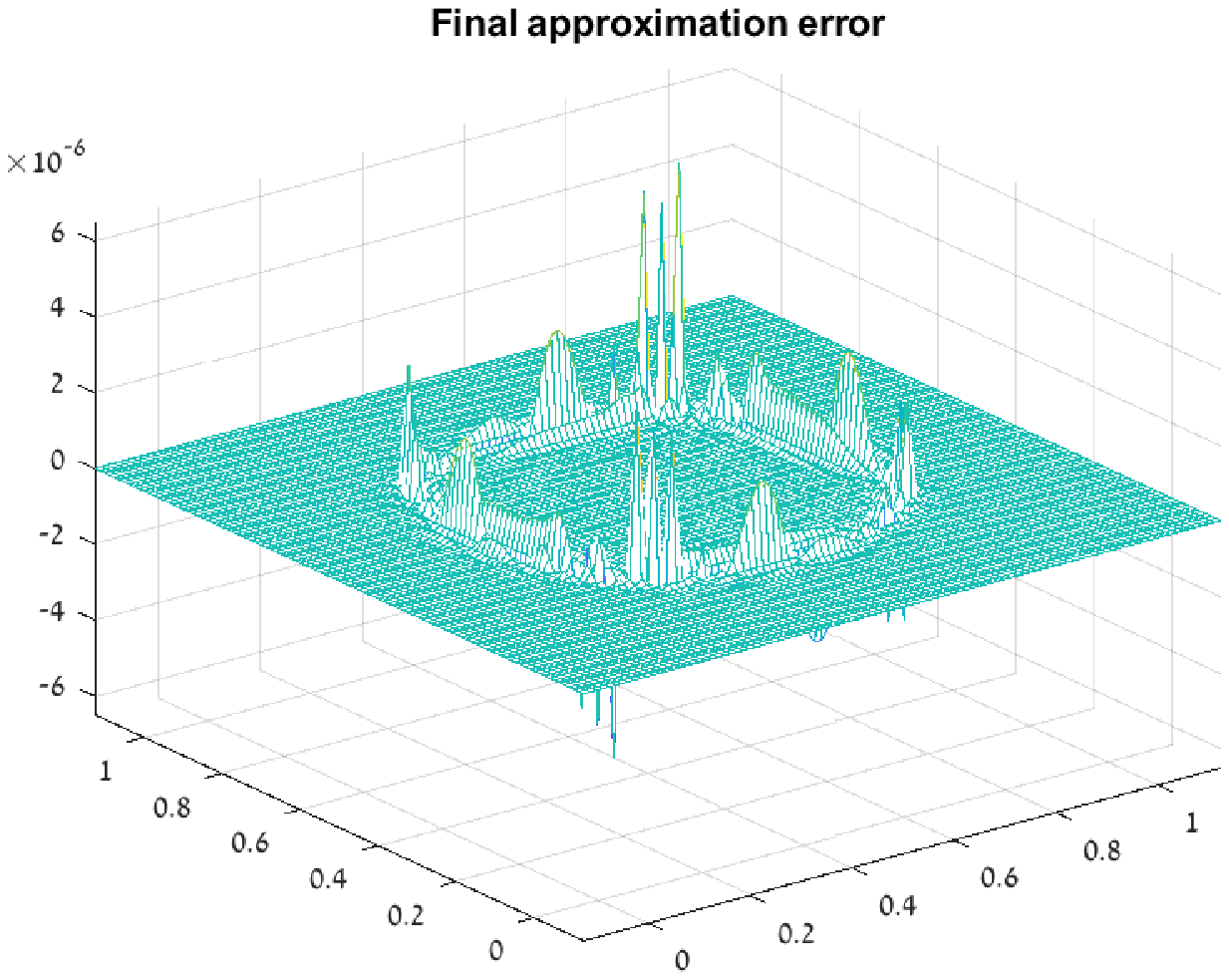}
    \caption{The final corrected approximation errors, at $z=0.5$.}
    \label{3D2ndapperrorB4}
\end{center}    
\end{figure}

\vfill\eject

\section*{Appendix}

\subsection*{From $k$th difference to lower order differences}\hfill

\medskip
\begin{lemma}
Let $h=\frac{1}{N+1}$,
$$ g(ih)=0, \ \ i<0,$$
and
$$\Delta^kg(ih)=O(h^\alpha),\ \ \ -k\le i \le N-k.$$
Then,
$$\Delta^{k-\ell}g(ih)=O(h^{\alpha-\ell}),\ \ \ -k\le i \le N-k+\ell.$$
\end{lemma}
\begin{proof}
It is enough to prove for one step, say from $\Delta^1$ to $\Delta^0$.

From $\Delta g(ih)=g((i+1)h)-g(ih)$ we have
$$g((i+1)h)=\Delta g(ih)+g(ih).$$
Assuming $g(-h)=0$ it follows that
$$g(jh)=\sum_{i=-1}^{j-1}\Delta g(ih).$$
Assuming $$\Delta g(ih)=O(h^\alpha),\ \ \ -1\le i \le N-1,$$
we obtain
$$|g(jh)|\le jCh^\alpha ,\ \ 0\le j \le N,$$
implying
$$g(jh)=O(h^{\alpha-1}).$$
Recursive application of the above result yieilds the result.
\end{proof}

\end{document}